 \documentclass[final,3p,times]{elsarticle}

\usepackage{amssymb}

\usepackage{amsfonts}
\usepackage{float}
\usepackage{comment}
\usepackage{amsmath}
\usepackage{color}
 \usepackage{amsthm}
  \usepackage{amsbsy}
 \newtheorem{thm}{Theorem}[section]
 
 \newtheorem{question}{Question}[section]
 
 \newtheorem{lem}[thm]{Lemma}

\journal{}

\begin{document}

\begin{frontmatter}



\title{Notes on $k$-rainbow independent domination in graphs}


\author[gzhu]{Enqiang Zhu}
\author[dg]{Chanjuan Liu\footnote{corresponding author}}

\address[gzhu]{Institute of Computing Science and Technology, Guangzhou University, Guangzhou 510006, China}
\address[dg]{School of Computer Science and Technology, Dalian University of Technology, Dalian 116024, China}

\begin{abstract}
The $k$-rainbow independent domination number of a graph $G$, denoted $\gamma_{\rm rik}(G)$, is the cardinality  of a smallest set consisting of two vertex-disjoint independent sets $V_1$ and $V_2$ for which every vertex in $V(G)\setminus (V_1\cup V_2)$ has neighbors in both $V_1$ and $V_2$.  This domination invariant was proposed by {\v{S}}umenjak, Rall and Tepeh in (Applied Mathematics and Computation 333(15), 2018: 353-361), which allows to reduce the problem of computing the independent domination number of the generalized prism $G {\Box}  K_k$ to an integer labeling problem on $G$. They proved a Nordhaus-Gaddum-type theorem: $5\leq \gamma_{\rm rik}(G)+\gamma_{\rm rik}(\overline{G})\leq n+3$ for every graph $G$ of order $n\geq  3$, where $\overline{G}$ is the complement of $G$. In this paper, we improve this result by showing that if $G$ is not isomorphic to the 5-cycle, then $5\leq \gamma_{\rm rik}(G)+\gamma_{\rm rik}(\overline{G})\leq n+2$. Moreover, we show that the problem of deciding whether a graph has a $k$-rainbow independent dominating function of a given weight is $\mathcal{NP}$-complete. Our results respond some open questions proposed by \v{S}umenjak, et al.
\end{abstract}


\end{frontmatter}

\section{Introduction}

All graphs considered in this paper are simple and for notation and terminology not defined here we follow the book \cite{bondy2008}. Let $G$ be a graph with \emph{vertex set} $V(G)$ and \emph{edge set} $E(G)$. Two vertices are \emph{adjacent} in $G$ if they are the endpoints of an edge of $G$. We say that a vertex  $u\in V(G)$ is \emph{adjacent} to a set $U\subseteq G$ in $G$ if $U$ contains a vertex adjacent to $u$ in $G$.
For any $v\in V(G)$, $N_G(v)$=$\{u|uv \in E(G)\}$  is called the \emph{open neighborhood} of $v$ in $G$  and $N_G[v]$=$N_G(v)\cup \{v\}$ is the \emph{closed neighborhood} of $v$ in $G$.
Let $d_G(v)=|N_G(v)|$ denote the degree of $v$ in $G$ and  $\Delta(G)=\max \{d_G(v)|v\in V(G)\}$.  A vertex of degree $k$ and at least $k$ is called a $k$-vertex and $k^+$-vertex, respectively. For any $S\subseteq V(G)$, let $N_G(S)$=$\bigcup_{v\in S}N_G(v)\setminus S$ and  $N_G[S]$=$N_G(S)\cup S$. We say that $S$ \emph{dominates} a set $S'$ if $S'\subseteq N_G[S]$. Moreover, we use the notation  $G-S$ to denote the subgraph of $G$ obtained by deleting vertices of $S$ and their incident edges in $G$, and $G[S]=G-(V(G)\setminus S)$ subgraph of $G$ induced by $S$.  The complete graph with $n$ vertices and the cycle of length $n$ are denoted by $K_n$ and $C_n$, respectively.
For two integers $i,j$, $i<j$, we will make use the notation $[i,j]$ to denote the set $\{i,i+1,\ldots, j\}$.

Given a graph $G$ and a subset $D\subseteq V(G)$, we call $D$ a dominating set of $G$ if $D$ dominates $V(G)$. An \emph{independent set} of a graph is a set of vertices no two of which are adjacent in the graph.  If a dominating set $D$ of $G$ is an independent set, then $D$ is called an \emph{independent dominating set} (IDS for short) of $G$. The \emph{independent domination number} of $G$, denoted by $i(G)$, is the cardinality of a smallest independent dominating set of $G$.
Domination and independent domination in graphs have always attracted extensive attention \cite{haynes1998,goddard2013-0}, and many variants of domination \cite{haynes1998} have been introduced increasingly, for the applications in diverse fields, such as electrical networks, computational biology, land surveying, etc. Recent studies on these variations include strong roman domination \cite{Ruiz2017}, sum of domination number \cite{Ruiz2017}, semitotal domination\cite{goddard2014,zhu2018}, relating domination \cite{Henning2017}, just to name a few.


Let $G\Box H$ be the cartesian product of $G$ and $H$. To reduce the problem of determining $i(G\Box K_k)$ to an integer labeling problem on $G$ itself, {\v{S}}umenjak et al. \cite{tadeja2018} recently proposed a new variation of domination, called \emph{$k$-rainbow independent dominating function} of a graph $G$ ($k$RiDF for short), which is a function $f$: $V(G) \rightarrow [0,k]$ such that $V_i$ is an independent set and every vertex $v$ with $f(v)=0$ is adjacent to a vertex $u$ with $f(u)=i$, for all $i\in [1,k]$. Alternatively, a $k$RiDF $f$ of a graph $G$ may be viewed as an ordered partition  $(V_0,V_1,\ldots, V_k)$ such that $V_j$ is an independent set for $j=0,1,\ldots, k$ and $N_G(x)\cap V_i\neq \emptyset$ for every $x\in V_0$ and each $i\in [1,k]$, where $V_j$ denotes the set of vertices assigned value $j$ under $f$. The \emph{weight} $w(f)$ of a $k$RiDF $f$ is defined as the number of nonzero vertices, i.e., $w(f)$=$|V(G)|-|V_0|$.
The \emph{$k$-rainbow independent domination number} of $G$, denoted by $\gamma_{\rm rik}(G)$, is the minimum weight of a $k$RiDF of $G$.
From the definition, we have $\gamma_{ri1}(G)=i(G)$. A $\gamma_{rik}(G)$-function is a $k$RiDF of $G$ with weight $\gamma_{rik}(G)$.


Let $G$ be  a graph  and $H$  a subgraph of $G$. Suppose that $g$ is a $k$RiDF of $H$. We say that a $k$RiDF $f$ of $G$ is \emph{extended} from $g$ if $f(v)=g(v)$ for every  $v\in V(H)$. In what follows, to prove that a graph $G$ has a $k$RiDF, we will first find a $k'$RiDF $g$ of a  subgraph $G'$ of $G$, $k'\leq k$ and then extend $g$ to a $k$RiDF $f$ of $G$. As for the remaining part of this paper, Section \ref{sec3} is dedicated to characterizing graphs  $G$ with $\gamma_{\rm ri2}(G)$=$|V(G)|-1$, based on which  we in Section \ref{sec4} show an improved Nordhaus-Gaddum-type theorem on the sum of 2-rainbow independent domination number of $G$ and its complement.  In Section \ref{sec2} we are devoted to the proof of $\mathcal{NP}$-completeness of the $k$-rainbow independent domination problem, and in the last section we give a conclusion of this paper.


\section{Graphs $G$ with $\gamma_{\rm ri2}(G)$=$|V(G)|-1$}
\label{sec3}

To get the improved Nordhaus-Gaddum-type theorem on the sum of 2-rainbow independent domination number of $G$ and of $\overline{G}$, we have to characterize the graphs $G$ such that $\gamma_{\rm ri2}(G)$=$|V(G)|-1$. For this, we need the following special graphs.

 A star $S_n$, $n\geq 1$, is a complete bipartite graph $G[X, Y]$ with $|X|$=1 and $|Y|=n$, where the vertex in $X$ is called the \emph{center} of $S_n$ and the vertices in $Y$ are \emph{leaves} of $S_n$.
The graph obtained from $S_n$  by adding a single edge is denoted $S_n^+$.
A \emph{double star} \cite{grossman1979} is defined as the union of two vertex-disjoint stars with an edge connecting their centers. Specifically, for two integers $n,m$ such that $n\geq m\geq 0$ the \emph{double star}, denoted by $S(n,m)$, is the graph with vertex set $\{u_0,u_1,\ldots, u_n, v_0,v_1, \ldots, v_m\}$
and edge set $\{u_0v_0, u_0u_i,v_0v_j|1\leq i\leq n, 1\leq j\leq m\}$, where $u_0v_0$ is called the \emph{bridge} of $S(n,m)$ and the subgraphs induced by $\{u_i|i=0,1,\ldots, n\}$ and $\{v_j|j=0,1,\ldots, m\}$ are called the \emph{$n$-star at $u_0$} and \emph{$m$-star at $v_0$}. Observe that $S(n,m)$ is defined on the premise of $n\geq m$. For mathematical convenience, we denote a double star $S(n,m)$ as a vertex-sequence $v_mv_{m-1}\ldots v_0u_0u_1\ldots u_n$.

We start with a known result which characterizes graphs $G$ with $\gamma_{\rm ri2}(G)=n$.

\begin{lem}\label{n} \cite{tadeja2018}
For any graph $G$ of order $n$,  $\gamma_{\rm ri2}(G)=n$ if and only if every connected component of $G$ is isomorphic either to $K_1$ or $K_2$. In addition, if $\gamma_{\rm ri2}(G)=n$, then $\gamma_{\rm ri2}(\overline{G})=2$, where $\overline{G}$ is the complement of $G$.
\end{lem}


The following conclusion is simple but will be used throughout this paper.

\begin{lem}\label{lemma3-3}
Let $G$ be a graph and $H$ a subgraph of $G$. Suppose that $g=(V_0,V_1,\ldots, V_k)$ is a  $\gamma_{\rm rik}(H)$-function. Then $g$ can be extended to a kRiDF of $G$ with weight at most $|V(G)|-|V_0|$.
\end{lem}
\begin{proof}
Let $V(G)\setminus V(H)=\{x_1,\ldots,x_{\ell}\}$. We will deal with these vertices in the order of $x_1,\ldots, x_{\ell}$ by the following rule: for each $x_i$, $i\in [1,\ell]$, let $j\in[1,k]$ be the smallest one such that $x_i$ is not adjacent to $V_j$ in $G$. If such $j$ does not exist, we update $V_0$ by $V_0\cup \{x_i\}$; otherwise we update $V_j$ by $V_j\cup \{x_i\}$. After the last one, i.e., $x_{\ell}$ is handled, we obtain a $k$RiDF of $G$. Obviously, the weight of the resulting \emph{k}RiDF of $G$ is at most $|V(G)|-|V_0|$. \qed
\end{proof}

The following theorem clarifies the structure of  connected graphs $G$ with  $\gamma_{\rm ri2}(G)=|V(G)|-1$.

\begin{thm} \label{n-1}
Let $G$ be a connected graph of order $n\geq 3$. Then, $\gamma_{\rm ri2}(G)=n-1$ if and only if $G$ is isomorphic to one among $S_{n-1}$, $S_{n-1}^+$, $S(n-3,1)$ ($n\geq 4$) and $C_5$.
\end{thm}

\begin{proof}
Let $f=(V_0,V_1,V_2)$ be an arbitrary $\gamma_{\rm ri2}(G)$-function. Observe that  $V_0$ does not contain any 1-vertex; one can readily derive that $\gamma_{\rm ri2}(G)=n-1$ when $G$ is isomorphic to one of $S_{n-1}, S_{n-1}^+$, $S(n-3,1)$ and $C_5$. Conversely, suppose that $\gamma_{\rm ri2}(G)=n-1$, i.e. $|V_0|=1$.  By Lemma \ref{lemma3-3}, $G$ contains no subgraph $H$ that has a 2RiDF of weight at most $|V(H)|-2$.
Since $\gamma_{\rm ri2}(C_4)=2=|V(C_4)|-2$,  $G$ contains no subgraph isomorphic to $C_4$. This also shows that every two vertices of $G$ share at most one neighbor in $G$.

\textbf{Observation 1.} \emph{If $G$ contains a $3^+$-vertex $x$, then every $2^+$-vertex of $G$ belongs to $N_G(x)$}. Suppose to the contrary that $G$ contains a $2^+$-vertex $y$ such that $y\notin N_G(x)$.  Let  $\{x_1,x_2,x_3\}\subseteq N_G(x)$ and $\{y_1,y_2\}\subseteq N_G(y)$. Observe that  $|\{x_1,x_2,x_3\}\cap \{y_1,y_2\}|\leq 1$ and $|N_G(y_i)\cap \{x_1,x_2,x_3\}|\leq 1$ for $i=1,2$; we without loss of generality assume that $y_2\notin \{x_1,x_2,x_3\}$, $y_2x_2\notin E(G)$ and $y_2x_3\notin E(G)$. Let $f$ be: $f(x)=f(y)=0, f(x_2)=1, f(x_3)=2$. Notice that either $y_1=x_j$ or $y_1x_j\notin E(G)$ for some $j\in [2,3]$; we further let $f(y_1)=f(x_j)$ and $f(y_2)=[1,2]\setminus \{f(y_1)\}$. Clearly, $f$ is a 2RiDF of $G[\{x,x_2,x_3,y,y_1,y_2\}]$ of weight  $|\{x,x_2,x_3,y,y_1,y_2\}|-2$, a contradiction.

\textbf{Observation 2.} \emph{$G$ contains at most one $3^+$-vertex}.  Suppose to the contrary that $G$ has two distinct $3^+$-vertices, say $x$ and $y$.
By Observation 1,  $xy\in E(G)$. Let $\{y,x_1,x_2\}\subseteq N_G(x)$ and $\{x,y_1,y_2\}\subseteq N_G(y)$. Since $G$ contains no subgraph isomorphic to $C_4$, $|\{x_1,x_2\}\cap \{y_1,y_2\}|\leq 1$ and there are no edges between $\{x_1,x_2\}$ and $\{y_1,y_2\}$. We assume that $x_2\notin \{y_1,y_2\}$ and $y_2\notin \{x_1,x_2\}$. Then, the function $f$: $\{x,x_1,x_2,y,y_1,y_2\} \rightarrow \{0,1,2\}$ such that  $f(x)$=$f(y)$=$0$, $f(x_2)$=$f(y_2)$=$2$ and $f(x_1)$=$f(y_1)$=$1$,
is a 2RiDF of $G[\{x,y,x_1,x_2,y_1,y_2\}]$ of weight $|\{x,y,x_1,x_2,y_1,y_2\}|-2$,  a contradiction.

\textbf{Observation 3.}  \emph{If $G$ contains a $3^+$-vertex $x$, then $N_G(x)$ contains at most two 2-vertices; in particular, when $N_G(x)$ contains two 2-vertices, these two 2-vertices are adjacent in $G$}. If not, suppose that $N_G(x)$ contains three 2-vertices, say $x_1,x_2,x_3$. Without loss of generality, we assume that $x_3\notin N_G(\{x_1,x_2\})$ and let $N_G(x_3)=\{x, y_3\}$.  Let $N_G(x_1)=\{x,y_1\}$ (possibly $y_1=x_2$, but $y_1\neq y_3$). By Observation 1, $d_G(y_3)=1$, i.e., $y_1y_3\notin E(G)$. Let $f$ be: $f(x_1)=f(x_3)=0, f(x)=1, f(y_1)=f(y_3)=2$. Obviously, $f$ is a 2RiDF of $G[\{x,x_1,y_1,x_3,y_3\}]$ of weight $|\{x,x_1,y_1,x_3,y_3\}|-2$, a contradiction.
Now, suppose that $N_G(x)$ contains two 2-vertices, say $x_1,x_2$. If $x_1x_2\notin E(G)$, let $N_G(x_i)=\{x, y_i\}, i=1,2$. Clearly, $y_1\neq y_2$ and  $y_1y_2\notin E(G)$.  Let $f$ be: $f(x_1)=f(x_2)=0, f(x)=1, f(y_1)=f(y_2)=2$. Then, $f$ is a 2RiDF of $G[\{x,x_1,y_1,x_2,y_2\}]$ of weight $|\{x,x_1,x_2,y_1,y_2\}|-2$, a contradiction.

By the above three observations and the assumption that $G$ is connected, we see that if $G$ contains a $3^+$-vertex $x$, then  $V(G)\setminus \{x\}$ contains either only  1-vertices ($G\cong S_{n-1}$), or one 2-vertex and $n-2$  1-vertices ($G\cong S(n-3,1)$), or two adjacent 2-vertices  and $n-3$ 1-vertices ($G\cong S_{n-1}^+$); if $\Delta(G)=2$, then it is easy to see that $G$ is isomorphic to one of  $S^+_2, S_2, S(1,1)$ and $C_5$. \qed
\end{proof}

According to Lemma \ref{n}, Theorem \ref{n-1} and the fact that   $\gamma_{\rm ri2}(G)$=$\sum_{i=1}^{k}\gamma_{\rm ri2}(G_i)$, where $G_1,\ldots, G_k$ are the components of $G$, we have the following theorem.

\begin{thm} \label{n-1-thm}
Let $G$ be a  graph of order $n\geq 3$. Then, $\gamma_{\rm ri2}(G)=n-1$ if and only if  $G$ has one component $G_1$  isomorphic to one among $S_{n_1-1}$ ($n_1\geq 3$), $S_{n_1-1}^+$ ($n_1\geq 3$), $S(n_1-3,1)$ ($n_1\geq 4$) and $C_5$, and other components are isomorphic to $K_1$ or $K_2$, where $n_1=|V(G_1)|$.
\end{thm}

\section{An improved Nordhaus-Gaddum type theorem for $\gamma_{\rm ri2}(G)$}
\label{sec4}
This section is devoted to achieve an improved Nordhaus-Gaddum type theorem by showing that $\gamma_{\rm ri2}(G)+\gamma_{\rm ri2}(\overline{G})\leq n+2$ for every graph $G\not\cong C_5$ of order $n\geq 2$, which improves a result obtained by \v{S}umenjak, et al \cite{tadeja2018}. Before doing so, we need to establish some simple lemmas.


\begin{lem}\label{n-1-lem}
Let $G$ be a graph of order $n\geq 3$. If $G$ is isomorphic to $S_{n-1}, S_{n-1}^+$ or $S(n-3,1)$, then  $\gamma_{\rm ri2}(\overline{G})\leq 3$.
\end{lem}

\begin{proof}
If $G\cong S_{n-1}$ or $G\cong S_{n-1}^+$, let $V(G)=\{v_0,v_1,\ldots, v_{n-1}\}$ where $v_0$ is the center and $v_1v_2\in E(G)$ when $G\cong S_{n-1}^+$. Define a function $f$ such that $f(v_1)=1, f(v_0)=f(v_2)=2$ and $f(v)=0$ for every $v\in V(\overline{G})\setminus \{v_0,v_1,v_2\}$. Since every vertex in $V(\overline{G})\setminus \{v_0,v_1,v_2\}$ is adjacent to both $v_1$ and $v_2$ in $\overline{G}$, it follows that $f$ is a 2RiDF of $\overline{G}$ of weight 3.

If $G\cong S(n-3,1)$, then $n\geq 4$. Let $G=v_1v_0u_0u_1\ldots u_{n-3}$. If $n=4$, then both $G$ and $\overline{G}$ are isomorphic to $P_4$, the path of length 3, and the conclusion holds. If $n\geq 5$, then the function $f$: $V(\overline{G})\rightarrow \{0,1,2\}$ such that $f(u_1)=f(u_0)=1, f(u_2)=2$ and $f(v)=0$ for every $v\in V(\overline{G})\setminus \{v_0,v_1,v_2\}$ is a 2RiDF of $\overline{G}$ with weight 3.\qed
\end{proof}

\begin{lem}\label{four}
Let $G$ be a graph of order $n$. If $G\not\cong C_5$ and $\gamma_{\rm ri2}(G)=4$, then $\gamma_{\rm ri2}(\overline{G})\leq n-2$.
\end{lem}

\begin{proof}
Clearly, $n\geq 4$. When $n=4$ and $n=5$, the conclusion is easy to prove and we assume that $n\geq 6$.
Suppose, to the contrary, that $\gamma_{\rm ri2}(\overline{G})\geq n-1$. If $\gamma_{\rm ri2}(\overline{G})=n$, then $\gamma_{\rm ri2}(G)=2$ by Lemma \ref{n}, a contradiction. Therefore, $\gamma_{\rm ri2}(\overline{G})=n-1$. By Theorem \ref{n-1-thm} $\overline{G}$ has one component isomorphic to $S_{n_1}, S_{n_1}^+$, $S(n_2,1)$ or $C_5$ where $n_1\geq 2, n_2\geq 1$, and all of the other components of $\overline{G}$ are isomorphic to $K_1$ or $K_2$.

If $\overline{G}$ contains two vertices $u,v$ such that $N_{\overline{G}}\{u,v\}=\emptyset$, then each of $u$ and $v$ is adjacent to all vertices of $V(G)\setminus \{u,v\}$ in $G$. We can obtain a 2RiDF of $G$ by assigning 1 to $u$, 2 to $v$, and 0 to the remained vertices of $G$. This indicates that $\gamma_{\rm ri2}(G)\leq 2$ and a contradiction. Therefore, $\overline{G}$ contains no $K_2$ components and contains at most one $K_1$ component, which implies that $\overline{G}$ has at most two components. If $\overline{G}$ contains only one component, then $\overline{G}$ is isomorphic to  $S_{n-1}, S_{n-1}^+$ or $S(n-3,1)$ (since $G\not \cong C_5$). By Lemma \ref{n-1-lem} $\gamma_{\rm ri2}(G)\leq 3$ and a contradiction.
Therefore,  $\overline{G}$ has two components, denoted by $G_1$ and $G_2$, where $G_1\cong K_1$ and $G_2$  is isomorphic to $S_{n-2}, S_{n-2}^+$, $S(n-4,1)$ or $C_5$. Let $V(G_1)=\{u\}$ and define a function $f$ as follows:
let $f(u)=1$; $f(v_0)=f(v')=2$ when $G_2\cong S_{n-2}$ or $G_2\cong S_{n-2}^+$ (where $v_0$ is the center of $G_2$ and $v'$ is a 1-vertex of $G_2$. Since $n\geq 6$ such $v'$ does exist), $f(v_0)=f(u_0)=2$ when $G_2\cong S(n-4,1)$ (where $v_0u_0$ is the bridge of $G_2$), or $f(u_1)=f(u_2)=2$ when $G_2\cong C_5$ (where $C_5=u_1u_2u_3u_4u_5u_1$); and all of the other remained vertices are assigned value 0. Clearly, every vertex assigned value 0 is adjacent to $u$ and a vertex assigned value 2. Hence, $f$ is a 2RiDF of $G$ with weight 3, and a contradiction. \qed
\end{proof}

\begin{lem}\label{add8-8}
Suppose that $G$ is a graph of order $n$ such that $\gamma_{\rm ri2}(G)\geq 4$ and $\gamma_{\rm ri2}(G)$+$\gamma_{\rm ri2}(\overline{G})$=$n+3$. Let $f=(V_0,V_1,V_2)$ be an arbitrary $\gamma_{\rm ri2}(G)$-function. Then,
\begin{enumerate}
  \item [(1)] If $|V_0|\geq 2$, then for any two vertices $u,v\in V_0$, there are no vertices $u_1,u_2,v_1,v_2$ such that  $\{u_1,u_2\}\in N_{\overline{G}}(u)$, $\{v_1,v_2\}\in N_{\overline{G}}(v)$ and $u_iv_i\notin E(\overline{G})$ for $i=1,2$, where $u_1\neq u_2, v_1\neq v_2$ but possibly $u_i=v_i$;
  \item [(2)] Suppose that $u, v$ be two arbitrary vertices of $V_0$. Then, $|N_{\overline{G}}(\{u,v\})|\geq 3$.
  \item [(3)]  $|V_i|\geq 2$ for $i=0,1,2$.
\end{enumerate}
\end{lem}

\begin{proof}
For (1), if the conclusion were false, let $g$ be: $g(u_i)=g(v_i)=i$ for $i=1,2$ and  $g(u)=g(v)=0$. Then, $g$ is a 2RiDF of $\overline{G}[\{u,v, u_1,v_1,u_2,v_2\}]$ with weight $|\{u,v, u_1,v_1,u_2,v_2\}|-2$. Since $V_1$ and $V_2$ are cliques in $\overline{G}$, $V_i$, for $i=1,2$, contains at most two vertices not assigned 0 under every 2RiDF of $\overline{G}$. Hence, by Lemma \ref{lemma3-3} $g$ can be extended to a 2RiDF of $\overline{G}$ with weight at most $|V_0|-2+4=|V_0|+2$. This shows that $\gamma_{\rm ri2}(\overline{G})\leq |V_0|+2$ and $\gamma_{\rm ri2}(G)$+$\gamma_{\rm ri2}(\overline{G})\leq |V_1|+|V_2|+|V_0|+2=n+2$, a contradiction.

For (2),   if $|N_{\overline{G}}(\{u,v\})|\leq 2$, let $f$ be: $f(u)=1,f(v)=2$ and $f(x)=0$ for $x\in V(G)\setminus N_{\overline{G}}[\{u,v\}]$. Clearly, $f$ is a 2RiDF of $G[V(G)\setminus N_{\overline{G}}(\{u,v\})]$ with weight 2. By Lemma \ref{lemma3-3}, $f$ can be extended to a 2RiDF of $G$ with weight at most $4$, since $|N_{\overline{G}}(\{u,v\})|\leq 2$. Thus, $\gamma_{\rm ri2}(G)=4$ and by Lemma \ref{four} $\gamma_{\rm ri2}(\overline{G})\leq n-2$,  a contradiction.

For (3), if $|V_0|=1$, then $\gamma_{\rm ri2}(G)$=$n-1$. By an analogous argument as that in Lemma \ref{four}, we can derive that $\gamma_{\rm ri2}(G)+\gamma_{\rm ri2}(\overline{G})\leq n+2$, a contradiction. In the following, we prove that $|V_1|\geq 2$ (the proof of $|V_2|\geq 2$ is similar to that of $|V_2|\geq 2$). Suppose that $|V_1|=1$ and let $V_1=\{u\}$. Then, every vertex of $V_0$ is adjacent to $u$ in $G$, i.e., $u$ is not adjacent to $V_0$ in $\overline{G}$. By Lemma \ref{four} we assume that $|V_1|+|V_2|\geq 5$.
If $V_0$ contains a vertex $v$ with two neighbors $v_1,v_2$ in $\overline{G}$,  then $u \notin \{v_1,v_2\}$.
Let $g$ be: $g(v)=0, g(v_1)=1,g(v_2)=2$. Since $V_2$ is a clique in $\overline{G}$, by Lemma \ref{lemma3-3} $g$ can be extended to a 2RiDF of $\overline{G}$ with weight at most $|V_0|-1+3=|V_0|+2$. This shows that $\gamma_{\rm ri2}(\overline{G})\leq |V_0|+2$ and hence $\gamma_{\rm ri2}(G)+\gamma_{\rm ri2}(\overline{G})\leq n+2$, a contradiction. Therefore, every vertex in $V_0$ has degree at most 1 in $\overline{G}$, which implies that $|N_{\overline{G}}(\{x,y\})|\leq 2$ for any two vertices $x\in V_0,y\in V_0$ (observe that $|V_0|\geq 2$). This contradicts (2).
\end{proof}

\begin{lem}\label{add1}
Let $G$ be a graph of order $n\geq 4$ and  $u\in V(G)$. If $H=G-u$, the resulting graph obtained from $G$ by deleting $u$ and its incident edges, is connected and $\gamma_{\rm ri2}(H)=|V(H)|-1$, then $G$ has a 2RiDF $f$ such that $f(u)=1$ and  $f(v)=0$ for some $v\in V(H)$.
\end{lem}
\begin{proof}
Clearly, $|V(H)|\geq 3$.  If $u$ is not adjacent to $V(H)$, then let $f$ be: $f(u)=1$ and $f(v)=g(v)$ for every $v\in V(H)$ where $g$ is a $\gamma_{\rm ri2}(H)$-function of $H$. Since $\gamma_{\rm ri2}(H)=|V(H)|-1$, there is a $v\in V(H)$ satisfying $f(v)=g(v)=0$. If $u$ is adjacent to a vertex $u_1\in V(H)$, then there is a vertex $u_2\in V(H)$ adjacent to $u_1$ since $H$ is connected.   Let $f$ be: $f(u_1)=0, f(u)=1, f(u_2)=2$. Then, by Lemma \ref{lemma3-3} $f$ can be extended to a desired 2RiDF of $G$. \qed
\end{proof}

Now, we turn to the proof of our main result.

\begin{thm}\label{mainresult}
Let $G$ be a graph of order $n(\geq 2)$. If $G\not\cong C_5$, then $\gamma_{\rm ri2}(G)+\gamma_{\rm ri2}(\overline{G})\leq n+2$.
\end{thm}

\begin{proof}
 We may assume that $n\geq 5$  as the statement holds trivially when $n=2,3,4$.
  Let $f_0=(V_0,V_1,V_2)$ be a  $\gamma_{\rm ri2}(G)$-function such that $\overline{G}[V_0]$ contains the maximum number of components isomorphic to $K_2$.
Suppose to the contrary that  $\gamma_{\rm ri2}(G)+\gamma_{\rm ri2}(\overline{G})>n+2$. Then,  $\gamma_{\rm ri2}(G)+\gamma_{\rm ri2}(\overline{G})=n+3$ since $\gamma_{\rm ri2}(G)+\gamma_{\rm ri2}(\overline{G})\leq n+3$ \cite{tadeja2018}, that is,
\begin{equation}\label{equation-0-0}
\gamma_{\rm ri2}(\overline{G})=|V_0|+3
\end{equation}
Formula (\ref{equation-0-0}) indicates that every 2RiDF of $\overline{G}$ has weight at least $|V_0|+3$. We will complete our proof by constructing  a 2RiDF of $\overline{G}$ of weight at most $|V_0|+2$ or  a 2RiDF of $G$ of weight less than $|V_1|+|V_2|$.

If $|V_1\cup V_2|=2$, then $\gamma_{\rm ri2}(G)+\gamma_{\rm ri2}(\overline{G})\leq 2+n$, a contradiction; if $|V_1\cup V_2|=3$, then $\gamma_{\rm ri2}(\overline{G})=n$ and by Lemma \ref{n} $\gamma_{\rm ri2}(G)=2$, also a contradiction. Therefore, by Lemma \ref{four},
\begin{equation}\label{equation-0}
|V_1|+|V_2|\geq 5
\end{equation}
Then,  by Lemma \ref{add8-8} (3) we have $|V_i|\geq 2$ for $i=0,1,2$.
In addition, because, by definition, $\overline{G}[V_i]$
is a clique, $i=1,2$, it follows that for every 2RiDF $g_0=(V'_0, V'_1, V'_2)$ of $\overline{G}$,
\begin{equation}\label{equation-1}
 |(V'_1\cup V'_2)\cap V_i|\leq 2, i=1,2
\end{equation}
Therefore,  by Lemma \ref{lemma3-3} every $\gamma_{\rm ri2}(\overline{G}[V_0])$-function can be extended to a 2RiDF of $\overline{G}$ with weight at most $\gamma_{\rm ri2}(\overline{G}[V_0])+4$, i.e., $\gamma_{\rm ri2}(\overline{G}[V_0])\geq |V_0|-1$ by Formula (\ref{equation-0-0}).

\textbf{Claim 1.} \emph{Let $\ell$ be the number of vertices in  $V_1\cup V_2$ which have degree $|V_1|+|V_2|-1$ in $\overline{G}[V_1\cup V_2]$. Then, $\ell \leq 1-\ell'$ where $\ell'=|V_0|-\gamma_{\rm ri2}(\overline{G}[V_0])$} $\leq 1$. If not,  either $\ell\geq 2$ or $\ell=\ell'=1$.
 Suppose that $\ell\geq 2$ and let $v_1$ and $v_2$ be two vertices of $V_1\cup V_2$ that are adjacent to all vertices of $(V_1\cup V_2)\setminus \{u,v\}$ in $\overline{G}$. Let $g'$ be: $g'(v_1)=1,g'(v_2)=2, g'(x)=0$ for $x\in V_1\cup V_2\setminus \{v_1,v_2\}$. Clearly, $g'$ is a 2RiDF of $\overline{G}[V_1\cup V_2]$ and by Lemma \ref{lemma3-3} $g'$ can be extended to a 2RiDF of $\overline{G}$ with weight at most $|V_0|+2$, a contradiction.  Now, suppose that $\ell=\ell'=1$. Then, $\gamma_{\rm ri2}(\overline{G}[V_0])=|V_0|-1$, which indicates that $\overline{G}[V_0]$ has a component $H'$ such that $\gamma_{\rm ri2}(H')=|V(H')|-1$. Since $\ell=1$, there is a vertex $v$, say $v\in V_1$, which is adjacent to all vertices of $V_2$ in $\overline{G}$. By Lemma \ref{add1} $\overline{G}[V(H')\cup \{v\}]$ has a  2RiDF $g'$ for which $g'(v)=1$ and $g'(x)=0$ for some $x\in V(H')$. Observe that every vertex in $(V_1\cup V_2)\setminus \{v\}$  is adjacent to $v$ in $\overline{G}$;
by the rule of Lemma \ref{lemma3-3} $g'$ can be extended to a 2RiDF $g$ of $\overline{G}$ under which there is at most one vertex in $V_1\setminus \{v\}$ (and $V_2$) not assigned value 0. Thus, $w(g)\leq |V_0|-1+3=|V_0|+2$, a contradiction. This completes the proof of Claim 1. $\blacksquare$

In the following, without loss of generality we assume $|V_1|\geq |V_2|$. Then, $|V_1|\geq 3$ by Formula (\ref{equation-0}).

\textbf{Claim 2.} \emph{$\overline{G}[V_0]$ contains no isolated vertex $v$ such that $N_{\overline{G}}(v)\cap V_1=\emptyset$}. Otherwise, let $f'$ be: $f'(v)=1, f'(x)=2$ for $x\in V_2$. By Claim 1, $V_1$ has at most one vertex adjacent to all vertices of $V_2$ in $\overline{G}$; say $v'$ if such a vertex exists.  We further let $f(y)=0$ for $y\in V_1\cup (V_0\setminus \{v\})$ (or for $y\in (V_1\setminus \{v'\})\cup (V_0\setminus \{v\})$ if $v'$ exists). Since  every vertex in $V_1\cup V_0$ (except for $v'$) is adjacent to both $v$ and $V_2$ in $G$,  $f$ is a 2RiDF of $G$ of weight at most $|V_2|+2$, a contradiction.  $\blacksquare$

We proceed by distinguishing two cases: $\gamma_{\rm ri2}(\overline{G}[V_0])=|V_0|-1$ and $\gamma_{\rm ri2}(\overline{G}[V_0])=|V_0|$.

\textbf{Case 1.} $\gamma_{\rm ri2}(\overline{G}[V_0])= |V_0|-1$. In this case,  by Claim 1  every vertex in $V_i$ has a neighbor in $V_j$ in $G$ where $\{i,j\}$=[1,2]; by Theorem \ref{n-1-thm}, $\overline{G}[V_0]$  has one component $H$ isomorphic to one of $S_{|V(H)|-1}$ ($|V(H)|\geq 3$), $S_{|V(H)|-1}^+$ ($|V(H)|\geq 3$), $S(|V(H)|-3,1)$ ($|V(H)|\geq 4$) and $C_5$, and other components of $\overline{G}[V_0]$  are isomorphic to $K_1$ or $K_2$. Let $u_0\in V(H)$ be a vertex with $d_H(u_0)=\Delta(H)$. Clearly, $d_{H}(u_0)\geq 2$. Let $u_1 \in N_H(u_0)$ and $u_2 \in N_H(u_0)$ be two vertices such that every vertex in $V(H)\setminus \{u_0,u_1,u_2\}$ has degree in $H$ not exceeding $\min \{d_H(u_1),d_H(u_2)\}$.
By the structure of $H$, for $i=1,2$, we have that $d_{H}(u_i)\leq 2$ and if $u_i$ has a neighbor $u'_i(\neq u_0$) in $H$, then $u_0u'_i\notin E(H)$.
Moreover, by Lemma \ref{add8-8} (1), $(N_{\overline{G}}(u_1)\cap N_{\overline{G}}(u_2))\setminus \{u_0\}=\emptyset$, which implies that every vertex in $V_1\cup V_2$ is adjacent to $u_1$ or $u_2$ in $G$.


\textbf{Claim 3.} \emph{$|V_0\setminus V(H)|\leq 1$}. Otherwise, let $\{v_1,v_2\}\subseteq (V_0\setminus V(H))$. Then, $d_{\overline{G}[V_0]}(v_1)\leq 1$ and $d_{\overline{G}[V_0]}(v_2)\leq 1$.   Suppose that $d_{\overline{G}[V_0]}(v_1)=1$ (the case of $d_{\overline{G}[V_0]}(v_2)=1$ can be similarly discussed). Let $v_1v'_1\in E(\overline{G}[V_0])$ and clearly $d_{\overline{G}[V_0]}(v'_1)=1$.
 By Lemma \ref{add8-8} (2), there exists a vertex $v_0\in (V_1\cup V_2)$ that is adjacent to $\{v_1,v'_1\}$ in $\overline{G}$. Without loss of generality, we assume that $v_1v_0\in E(\overline{G})$. By Lemma \ref{add1},  $\overline{G}[V(H)\cup \{v_0\}]$ has a 2RiDF $g'$ such that $g'(v_0)=1$ and $g'(x)=0$ for some $x\in V(H)$. Further, let $g'(v_1)=0$ and  $g'(v'_1)=2$. Then, $g'$ is a 2RiDF of $\overline{G}[V(H)\cup\{v_0,v_1,v'_1\}]$, and by Lemma \ref{lemma3-3} and Formula (\ref{equation-1})  $g'$ can be extended to a 2RiDF  of $\overline{G}$ with weight at most $|V_0|-2+4=|V_0|+2$ (since $g'(v_1)=g'(x)=0$),  a contradiction.   We therefore assume that $d_{\overline{G}[V_0]}(v_1)=d_{\overline{G}[V_0]}(v_2)=0$. By Lemma \ref{add8-8} (2) we have $|N_{\overline{G}}\{v_1,v_2\}\cap (V_1\cup V_2)|\geq 3$. Without loss of generality, we may assume that $v_1$ is adjacent to two vertices of $V_1\cup V_2$ in $\overline{G}$, say $v_{11}$ and $v_{12}$. By Lemma \ref{add8-8} (1), $u_i$ is not adjacent to both $v_{11}$ and $v_{12}$, and $v_{1j}$ is not adjacent to both $u_1$ and $u_2$ in $\overline{G}$, where $i\in [1,2]$ and $j\in [1,2]$.  Thus, it follows that  $u_1v_{11}\notin E(\overline{G})$ and  $u_2v_{12}\notin E(\overline{G})$, or $u_1v_{12}\notin E(\overline{G})$ and  $u_2v_{11}\notin E(\overline{G})$, which contradicts to Lemma \ref{add8-8} (1) again.  $\blacksquare$

By  Claim 3, we see that $\overline{G}[V_0]$ contains no component isomorphic to $K_2$ and contains at most one $K_1$ component.

\textbf{Claim 4.} \emph{$\overline{G}[V_0]$ contains a $K_1$ component}.
If not, we have $\overline{G}[V_0]$=$H$.

\textbf{Claim 4.1.} $(N_{\overline{G}}(u_1)\cup N_{\overline{G}}(u_2))\cap (V_1\cup V_2)\neq \emptyset$. Otherwise, both $u_1$ and $u_2$ are adjacent to all vertices of $V_1\cup V_2$ in $G$, and by Lemma \ref{add8-8} (2)
$d_{H}(u_i)=2$ for $i=1,2$ and $u_1u_2\notin E(\overline{G})$. Let $\{u'_i\}=N_{H}(u_i)\setminus \{u_0\}, i=1,2$; then, $u_0u'_i\notin E(\overline{G})$. Let $f$ be: $f(u_1)=f(u'_1)=1$, $f(u_2)=f(u'_2)=2$ and $f(x)=0$ for $x\in V(G)\setminus \{u_1,u'_1,u_2,u'_2\}$. Then, $f$ is a 2RiDF  of $G$ with weight 4, a contradiction.   $\blacksquare$


\textbf{Claim 4.2.} \emph{$|V_1|=3$}. Observe that $|V_1|\geq 3$; it is enough to show that $G$ has a 2RiDF $f$ with $w(f)\leq |V_2|+3$.
When $u_1u_2\in E(\overline{G})$, let $f$ be: $f(u_0)=f(u_1)=f(u_2)=1, f(x)=0$ for $x\in (V_1\cup V_0)\setminus \{u_0,u_1,u_2\}$ and $f(y)=2$ for $y\in V_2$. By Lemma \ref{add8-8} (1), $V_1\cup V_0$ contains no vertex adjacent to both $u_1$ and $u_2$ in $\overline{G}$. Therefore, $f$ is a 2RiDF of $G$ of weight $|V_2|+3$. Now, suppose that $u_1u_2\notin E(\overline{G})$.
By Lemma \ref{add8-8} (1), $V_1$ contains at most one vertex adjacent to both $u_0$ and $u_1$ in $\overline{G}$; say $u$ if such a vertex exists. Let $f$ be: $f(u_0)=f(u_1)=1$ (or $f(u_0)=f(u_1)=f(u)=1$ if $u$ exists), $f(x)=0$ for $x\in (V_1\cup (V_0\setminus \{u_0,u_1\}))$ (or $x\in (V_1\cup V_0)\setminus \{u_0,u_1,u\}$) and $f(y)=2$ for $y\in V_2$.
 Notice that by Claim 1 every vertex in $V_0\cup V_1$ is adjacent to $V_2$ in $G$, and  by the structure of $H$ and the selection of $u_1$ and $u_2$, every vertex of $(V_0\cup V_1)\setminus \{u_0,u_1,u\}$ is adjacent to $\{u_0,u_1\}$ in $G$; $f$ is a 2RiDF of $G$ of weight at most $|V_2|+3$.
  $\blacksquare$

By Claim 4.2, we have $2\leq |V_2|\leq 3$. Let $V_1=\{w_1,w_2,w_3\}$ in the following.

\textbf{Claim 4.3.}  \emph{Every vertex of $V_i$ is adjacent to at most one vertex of $V_j$ in $\overline{G}$ for $\{i,j\}=[1,2]$}. If not, suppose that $V_2$ contains a vertex $v$ adjacent to two vertices of $V_1$ in $\overline{G}$, say $w_1,w_2$.
By Lemma \ref{add8-8} (1) $v$ is not adjacent to $u_1$ or $u_2$ in $\overline{G}$, say $u_1v\notin E(\overline{G})$.
If $u_2w_3\notin E(\overline{G})$, let $g'$ be:  $g'(u_i)=i$ for $i=0,1,2$, $g'(w_1)=g'(w_2)=0, g'(w_3)=2$, $g'(v)=1$. If $u_2w_3\in E(\overline{G})$, then $u_1w_3\notin E(\overline{G})$ and let $g'$ be: $g'(u_1)=g'(w_3)=1, g'(w_1)=g'(w_2)=0, g'(v)=2$; further, let $g'(u_2)=0$ when $u_2v\in E(\overline{G})$, or let $g'(u_2)=2$ and $g'(u_0)=0$ when $u_2v\notin E(\overline{G})$. By Lemma \ref{lemma3-3} in either case we can extended the $g'$ defined above to a 2RiDF $g$ of $\overline{G}$ under which
$g(w_1)=g(w_2)=0$ and $g(u_0)=0$ or $g(u_2)=0$. Therefore, by Formula (\ref{equation-1}) $w(g)\leq |V_0|-1+3=|V_0|+2$, a contradiction. With a similar argument, we can also get a contradiction if we assume  $V_1$ contains a vertex adjacent to two vertices of $V_2$ in $\overline{G}$.  $\blacksquare$


Now, we  consider the value of $|V_2|$. Suppose that $|V_2|=3$ and let $V_2=\{w_4,w_5,w_6\}$. By Claim 4.1, we may assume, without loss of generality, that $u_1w_1\in E(\overline{G})$. This indicates that $u_2w_1\notin E(\overline{G})$ by Lemma \ref{add8-8} (1).  If $u_2$ is adjacent to $V_2$, say $u_2w_4\in E(\overline{G})$,  then by Lemma \ref{add8-8} (1), $u_1w_4\notin E(\overline{G})$, $w_1w_4\in E(\overline{G})$, and  $u_1$ (resp. $u_2$) is not adjacent to $\{w_2,w_3\}$ (resp. $\{w_5,w_6\}$) in $\overline{G}$ (otherwise $w_4$ or $w_1$ is adjacent to two vertices of $V_1$ or $V_2$ in $\overline{G}$, respectively. This contradicts to Claim 4.3). Let $f$ be: $f(u_1)=f(w_1)=1, f(u_2)=f(w_4)=2$ and $f(x)=0$ for $x\in V(G)\setminus \{u_1,u_2,w_1,w_4\}$.
Observe that $w_1$ (resp. $w_4$) is not adjacent to $\{w_5,w_6\}$ (resp. $\{w_2,w_3\}$) in $\overline{G}$ and by Lemma \ref{add8-8} (1) $V_0\setminus \{u_0,u_1,u_2\}$ contains no vertex adjacent to both $u_i$ and $w_i$ for some $i\in [1,2]$.  Hence, $f$ is a 2RiDF of $G[V(G)\setminus \{u_0\}]$ of weight 4 and by Lemma \ref{lemma3-3} $f$ can be extended to a 2RiDF of $G$ with weight at most $5<|V_1|+|V_2|$, a contradiction.
Therefore, we may assume that $N_{\overline{G}}(u_2)\cap V_2=\emptyset$. In this case, when $N_{\overline{G}}(u_2)\cap V_1=\emptyset$, let $f$ be: $f(u_2)=2$, $f(u_0)=f(u_1)=1$. By Lemma \ref{add8-8} (1) $V_1\cup V_2$ contains at most one vertex $w'$ adjacent to both $u_0$ and $u_1$ in $\overline{G}$ and $V_0\setminus \{u_0\}$ contains at most one vertex $u'$ adjacent to $u_2$ in $\overline{G}$; we further let $f(x)=0$ for $x\in V(G)\setminus \{u_0,u_1,u_2,u',w'\}$. Then, $f$ is a 2RiDF of $G[V(G)\setminus \{u',w'\}]$ of weight 3 and by Lemma \ref{lemma3-3} $f$ can be extended to a  2RiDF of $G$ of weight at most $5<|V_1|+|V_2|$, a contradiction. We therefore suppose that  $u_2$ is adjacent to $V_1$ in $\overline{G}$, say $u_2w_2\in E(\overline{G})$. Then, with the same argument as $N_{\overline{G}}(u_2)\cap V_2=\emptyset$, we can show that $N_{\overline{G}}(u_1)\cap V_2=\emptyset$ as well.

Then, if $w_3u_1\notin E(\overline{G})$  and $w_3u_2\notin E(\overline{G})$, the function  $f$: $f(u_1)=f(w_1)=1, f(u_2)=f(w_4)=2$ and $f(x)=0$ for $x\in V(G)\setminus \{u_1,u_2,w_1,w_4, u_0\}$ is a 2RiDF of $G[V(G)\setminus \{u_0\}]$ with weight 4, and by Lemma \ref{lemma3-3} $f$ can be extended  to a 2RiDF of $G$ with weight at most $5<|V_1|+|V_2|$,  a contradiction. Therefore, we suppose that $w_3u_1\in E(\overline{G})$ by the symmetry. By Lemma \ref{add8-8} (1), it has that  $w_3u_2\notin E(\overline{G})$, and $u_0w_1\notin E(\overline{G})$ or  $u_0w_3\notin E(\overline{G})$, say $u_0w_1\notin E(\overline{G})$ by the symmetry.  Let $f$ be: $f(u_0)=f(u_1)=1, f(u_2)=f(w_2)=2$ and $f(x)=0$ for $x\in V(G)\setminus \{u_1,u_2,u_0, w_2,w_3\}$. Since every vertex in $V(G)\setminus \{u_1,u_2,u_0, w_2,w_3\}$ is adjacent to both $\{u_0,u_1\}$ and $\{u_2,w_2\}$ in $G$, $f$ is a 2RiDF of $G[V(G)\setminus \{w_3\}]$ of weight 4 and by Lemma \ref{lemma3-3} $f$ can be extended to a 2RiDF of $G$ of weight at most $5<|V_1|+|V_2|$, and  a contradiction.

A similar line of thought leads to a contradiction if we assume that $|V_2|=2$ and  proves Claim 4. $\blacksquare$

By Claim 4, we see  that $\overline{G}[V_0]$  contains one component isomorphic to $K_1$. Let $s$ be the vertex of the $K_1$ component. We first show that  $|N_{\overline{G}}(s)\cap (V_1\cup V_2)|\leq 1$. If not, we assume that $s$ is adjacent to two vertices of  $V_1\cup V_2$ in $\overline{G}$, say $s_1,s_2$.  By Lemma \ref{add8-8} (1)
$s_i$ (resp. $u_j$) is not adjacent to both $u_1$ and $u_2$ (resp. $s_1$ and $s_2$) in $\overline{G}$ for every $i,j\in [1,2]$. This implies that either  $s_iu_i\notin E(\overline{G})$ for $i=1,2$ or $s_1u_2\notin E(\overline{G})$ and $s_2u_1\notin E(\overline{G})$, which contradicts to Lemma \ref{add8-8} (1) as well. Thus, by Claim 2 $|N_{\overline{G}}(s)\cap (V_1\cup V_2)|=1$
 and the vertex $s'$  adjacent to $s$ in $\overline{G}$ belongs to $V_1$.
Let $f$ be: $f(s)=2$, $f(x)=1$ for $x\in V_1$, $f(y)=0$ for $y\in V_2\cup V(H))$. Observe that by Claim 1 every vertex in $V_2$ is adjacent to $V_1$ in $G$ and hence every vertex in $V_2\cup V(H)$ is adjacent to both $V_1$ and $s$ in $G$; $f$ is a 2RiDF of $G$ with weight $|V_1|+1<|V_1|+|V_2|$ (since $|V_2|\geq 2$), a contradiction.

The foregoing discussion shows that there exists a contradiction if we assume that  $\gamma_{\rm ri2}(\overline{G}[V_0])=|V_0|-1$. In the following, we consider the case of  $\gamma_{\rm ri2}(\overline{G}[V_0])= |V_0|$.

\textbf{ Case 2.} $\gamma_{\rm ri2}(\overline{G}[V_0])= |V_0|$. Then by Lemma \ref{n} each component of $\overline{G}[V_0]$ is isomorphic to $K_1$ or $K_2$. Recall that $|V_i|\geq 2$ for $i=0,1,2$.  Let $u,v$ be two vertices of $V_0$ such that  $uv\in E(\overline{G})$ if $\overline{G}[V_0]$ contains a $K_2$ component and  $u,v$ are isolated vertices in $\overline{G}[V_0]$ otherwise. By Lemma \ref{add8-8} (1), we have
\begin{equation}\label{0}
|(N_{\overline{G}}(u)\cap N_{\overline{G}}(v))\cap (V_1\cup V_2)|\leq 1
\end{equation}

We deal with two subcases in terms of the adjacency property of $u$ and $v$.

\textbf{Case 2.1.} $uv\in E(\overline{G})$. Then every vertex in $V_0\setminus \{u,v\}$ is not adjacent to $\{u,v\}$ in $\overline{G}$.

\textbf{Claim 5.} \emph{Every vertex in $V_1\cup V_2$ has degree at most $|V_1|+|V_2|-2$ in $\overline{G}[V_1\cup V_2]$}. Suppose that $V_1$ contains a vertex $w$ adjacent to all vertices of $V_2$ in $\overline{G}$. If $uw\in E(\overline{G})$  (or $vw\in E(\overline{G})$), then by Lemma \ref{lemma3-3} the 2RiDF $g'$ of $\overline{G}[\{u,v,w\}]$ such that $g'(u)=0$ (or $g'(v)=0), g'(w)=1$ and $g'(v)=2$ ($g'(u)=2$) can be extended to a 2RiDF of $\overline{G}$, under which $(V_1\cup V_2)\setminus \{w\}$ contains at most two vertices not assigned 0. Thus, $w(g)\leq |V_0|-1+3=|V_0|+2$, a contradiction. We therefore assume that $uw\notin E(\overline{G})$ and $vw\notin E(\overline{G})$. By Lemma \ref{add8-8} (2), $V_1\cup V_2$ contains at least three vertices adjacent to $u$ or $v$. Without loss of generality, we may suppose that there is a vertex $u'\in V_1\cup V_2$ that is adjacent to $u$ in $\overline{G}$. Define a 2RiDF $g'$ of $\overline{G}[\{u,v,u',w\}]$ as follows: $g'(u')=2, g'(u)=0$ and  $g'(v)=g'(w)=1$. Then, 
by Lemma \ref{lemma3-3} $g'$ can be extended to a 2RiDF $g$ of $\overline{G}$, under which  $(V_1\cup V_2)\setminus \{w,u'\}$ contains at most one vertex not assigned value 0. Therefore, $w(g)\leq |V_0|-1+3=|V_0|+2$, a contradiction.  With a similar argument, we can also obtain a contradiction if we assume that $V_2$ contains a vertex adjacent to all vertices of $V_1$. This completes the proof of Claim 5. $\blacksquare$

By Claim 5,  every vertex in $V_i$ has a neighbor in $V_j$ in $G$ for $\{i,j\}$=[1,2]. If $V_1\cap (N_{\overline{G}}(u)\cap N_{\overline{G}}(v))=\emptyset$, then every vertex in $V_1$ is adjacent to $u$ or $v$ in $G$.  Let $f$ be: $f(u)=f(v)=1, f(x)=2$ for $x\in V_2$ and $f(y)=0$ for $y\in V_1\cup (V_0\setminus \{u,v\})$. Clearly, $f$ is a 2RiDF of $G$ with  weight $|V_2|+2<|V_1|+|V_2|$, a contradiction. We therefore assume that $V_1$ contains a vertex $s$ such that $su\in E(\overline{G})$ and $sv\in E(\overline{G})$. Then, by Lemma \ref{add8-8} (1) $V_2\cup (V_1\setminus \{s\})$ contains no vertex adjacent to both $u$ and $v$ in $\overline{G}$. Analogously, the function $f$ such that $f(u)=f(v)=1, f(x)=2$ for $x\in V_1$ and $f(y)=0$ for $y\in V_2\cup (V_0\setminus \{u,v\})$ (and $f(u)=f(v)=f(s)=1, f(x)=2$ for $x\in V_2$ and $f(y)=0$ for $y\in (V_1\setminus \{s\})\cup (V_0\setminus \{u,v\})$) is a 2RiDF of $G$ with weight $|V_1|+2$ (and $|V_2|+3$). This implies that $|V_1|=3$ and $|V_2|=2$.
Let $V_1=\{s, s_1,s_2\}$ and $V_2=\{s_3,s_4\}$.¡¡Then, $\{u,v\}$ contains no vertex  adjacent to both $s_1$ and $s_2$ in $\overline{G}$; otherwise, we, by the symmetry, suppose that  $us_1\in E(\overline{G})$ and $us_2\in E(\overline{G})$. Then, the function $g'(u)=1,g'(s)=2, g'(v)=g'(s_1)=g'(s_2)=0$ is a 2RiDF of $\overline{G}[\{u,v,s,s_1,s_2\}]$ with weight 2, and by Lemma \ref{lemma3-3} $g'$ can be extended to a 2RiDF of $\overline{G}$ with weight at most $|V_0|-1+|V_2|+1=|V_0|+2$, a contradiction. In addition, by Lemma \ref{add8-8} (1) $s_i,i=1,2$, is not adjacent to both $u$ and $v$ in $\overline{G}$.
Therefore, we may assume, by the symmetry, that $s_1v\notin E(\overline{G})$ and  $s_2u\notin E(\overline{G})$.

Suppose that there are no edges between $\{u,v\}$ and $V_2$ in $\overline{G}$. By Lemmas \ref{add8-8} (2), $us_1\in E(\overline{G})$ and $vs_2\in E(\overline{G})$. Then, the function $g'$ such that $g'(u)=1,g'(s_2)=2, g'(s)=g'(s_1)=g'(v)=0$ is a 2RiDF of $\overline{G}[\{u,v,s,s_1,s_2\}]$ with weight 2. By Lemma \ref{lemma3-3} $g'$ can be extended to a 2RiDF of $\overline{G}$ with weight at most $|V_2|+1+|V_0|-1=|V_0|+2$, a contradiction. We therefore assume that there is an edge between $\{u,v\}$ and $V_2$ in $\overline{G}$, say $vs_3\in E(\overline{G})$ by the symmetry.

If  $s_4s\in E(\overline{G})$, then the function $g'$ such that $g'(s_3)=2, g'(s_4)=0, g'(s)=1,g'(v)=0$ is a 2RiDF of $\overline{G}[\{s,v, s_3,s_4\}]$ with weight 2, and by Lemma \ref{lemma3-3} and Formula \ref{equation-1} $g'$ can be extended to a 2RiDF of $\overline{G}$ of weight at most $|V_0|-1+3=|V_0|+2$, a contradiction. Consequently, we have $s_4s\notin E(\overline{G})$.  Then, the function $g'$ such that $g'(s_3)=0, g'(s_4)=g'(s)=2, g'(v)=1,g'(u)=0$ is a 2RiDF of $\overline{G}[\{s,u,v, s_3,s_4\}]$ with weight 3, and by Lemma \ref{lemma3-3} and Formula \ref{equation-1} $g'$ can be extended to a 2RiDF of $\overline{G}$ with weight at most $|V_0|-1+3=|V_0|+2$, a contradiction.

\textbf{Case 2.2.} $uv\notin E(\overline{G})$. Then, by the selection of $u,v$ and $f_0$, $\overline{G}[V_0]$ contains only isolated vertices and $G$ contains no $\gamma_{\rm ri2}(G)$-function for which the subgraph of $\overline{G}$ induced by the set of vertices assigned value 0 contains $K_2$ components.

For every $x\in V_0$, let $U_i^x=N_{\overline{G}}(x)\cap V_i$ for $i=1,2$. Let $f'$ be: $f'(u)=1,f'(v)=2$ and $f'(x)=0$ for $x\in ((V_1\cup V_2)\setminus (U^u_1\cup U^u_2\cup U^v_1 \cup U^v_2))\cup (V_0\setminus \{u,v\})$. Obviously, $f'$ is a 2RiDF of $G-(U^u_1\cup U^u_2\cup U^v_1 \cup U^v_2))$ with  weight 2.
By Lemma \ref{lemma3-3} $f'$ can be extended to a 2RiDF of $G$ with weight at most $|(U^u_1\cup U^u_2\cup U^v_1 \cup U^v_2))|+2$. To ensure   $|(U^u_1\cup U^u_2\cup U^v_1 \cup U^v_2))|+2\geq |V_1|+|V_2|$, we have

\begin{equation}\label{equa1}
|(V_1\cup V_2)\setminus (U^u_1\cup U^u_2\cup U^v_1 \cup U^v_2)|\leq 2
\end{equation}

\textbf{Claim 6.} \emph{$|(V_1\cup V_2)\setminus (U^u_1\cup U^u_2\cup U^v_1 \cup U^v_2)|=2$ and the two vertices in $(V_1\cup V_2)\setminus (U^u_1\cup U^u_2\cup U^v_1 \cup U^v_2)$ are adjacent in $\overline{G}$}.
Let $g'$ be a 2RiDF of $\overline{G}[V_0]$ such that $g'(u)=g'(v)=1$.
Suppose that $|(V_1\cup V_2)\setminus (U^u_1\cup U^u_2\cup U^v_1 \cup U^v_2)|\leq 1$.
Since $V_1$ and $V_2$ are cliques in $\overline{G}$ and every vertex in $U^u_1\cup U^u_2\cup U^v_1 \cup U^v_2$ is adjacent to $u$ or $v$ in $\overline{G}$, by Lemma \ref{lemma3-3} $g'$ can be extended to a 2RiDF $g$ of $\overline{G}$ under which  at most one vertex in $V_i, i=1,2$, is not assigned value 0 (here if  $(V_1\cup V_2)\setminus (U^u_1\cup U^u_2\cup U^v_1 \cup U^v_2)$ contains a vertex, say $w$, then let $g(w)=2$). Clearly, $w(g)=w(g')+2\leq |V_0|+2$, a contradiction. Moreover, if $(V_1\cup V_2)\setminus (U^u_1\cup U^u_2\cup U^v_1 \cup U^v_2)$ contains two nonadjacent vertices in $\overline{G}$, say $w_1,w_2$, then $w_1$ and $w_2$ are not in the same set $V_i$ for some $i\in [1,2]$. Therefore, we can extend $g'$ to a 2RiDF $g$ of $\overline{G}$ by letting $g'(w_1)=g'(w_2)=2$ and $g'(x)=0$ for $x\in (V_1\cup V_2)\setminus \{w_1,w_2\}$. But $w(g)=w(g')+2 \leq |V_0|+2$,  a contradiction. $\blacksquare$

By Claim 6, $(V_1\cup V_2)\setminus (U^u_1\cup U^u_2\cup U^v_1 \cup U^v_2)$ contains two adjacent vertices  in $\overline{G}$, say $w_1,w_2$.  If $V_0\setminus \{u,v\}$ contains a vertex $z$ that is adjacent to $w_1$ (or $w_2$) in $\overline{G}$,  then let $g'$ be: $g'(u)=g'(v)=g'(z)=1$, $g'(w_1)=0$ (or $g'(w_2)=0$), $g'(w_2)=2$ (or $g'(w_1)=2$). Since every vertex $(V_1\cup V_2)\setminus \{w_2\}$ is adjacent to $\{z,u,v\}$ in $\overline{G}$ and every vertex in $V'\setminus \{w_2\}$ is adjacent to $w_2$ where $w_2\in V'$ for some $V'\in \{V_1,V_2\}$,   by Lemma \ref{lemma3-3}  $g'$ can be extended to a 2RiDF $g$ of $\overline{G}$ under which every vertex in $V'\setminus \{w_2\}$ is assigned value 0 and at most one vertex in  $\{V_1,V_2\}\setminus V'$ is not assigned value 0. Therefore, $w(g)\leq |V_0|+2$, a contradiction. This shows that every vertex in $V_0$ is not adjacent to  $\{w_1,w_2\}$ in $\overline{G}$.
Furthermore, if there exists a vertex $z\in V_0\setminus \{u,v\}$, then by Claim 6 we have  $(V_1\cup V_2)\setminus (U^u_1\cup U^u_2\cup U^z_1 \cup U^z_2)=\{w_1,w_2\}$ and $(V_1\cup V_2)\setminus (U^v_1\cup U^v_2\cup U^z_1 \cup U^z_2)=\{w_1,w_2\}$, which implies that $N_{\overline{G}}(z)=U^u_1\cup U^u_2\cup U^v_1 \cup U^v_2$. Then, the function $g'$ such that $g'(z)=1, g'(u)=g'(v)=2$ and $g'(x)=0$ for $x\in U^u_1\cup U^u_2\cup U^v_1 \cup U^v_2$ is a 2RiDF of $\overline{G}-(\{w_1,w_2\}\cup (V_0\setminus \{u,v,z\}))$ with  weight 3, and by Lemma \ref{lemma3-3} $g'$ can be extended to a 2RiDF of $\overline{G}$ with weight at most $(|V_0|+2-3)+3=|V_0|+2$, a contradiction. So far, we have shown that $V_0=\{u,v\}$, i.e., $\gamma_{\rm ri2}(G)=n-2$.

Now, let $f'$ be: $f'(u)=1, f'(v)=2$ and $f'(w_1)=f'(w_2)=0$. Clearly, $f'$ is a 2RiDF of $G[\{u,v,w_1,w_2\}]$. Then, by Lemma \ref{lemma3-3} $f'$ can be extended to a 2RiDF $f$ of $G$ with weight at most $n-2$. To ensure $w(f)\geq \gamma_{\rm ri2}(G)=n-2$, we have $w(f)=n-2$, i.e., $f$ is a $\gamma_{\rm ri2}(G)$-function.
 We see that the subgraph of $\overline{G}$ induced by $\{w_1w_2\}$ is isomorphic to $K_2$. But this contradicts the selection of $f_0$. Eventually, we complete the proof of Theorem \ref{mainresult}.  \qed
\end{proof}

\section{The $\mathcal{NP}$-completeness}
\label{sec2}
In this section, we study the $\mathcal{NP}$-completeness of the $k$-rainbow independent domination problem. To prove a given problem $P$ to be $\mathcal{NP}$-complete, we have to show that $P\in \mathcal{NP}$ and find a known $\mathcal{NP}$-complete problem that can be reduced to $P$ in polynomial time. Here, by establishing an equivalence relation between the domination  problem and $k$-rainbow independent domination problem, we can show that the $k$-rainbow independent domination problem is $\mathcal{NP}$-complete when restricted to bipartite graphs. Three problems involved in our proof are described as follows:

\vspace{0.05cm}

\textbf{The independent domination problem (IDP)} \cite{manlove1999}.

\emph{Input}: A graph $G$ and a positive integer $k$;

\emph{Property}: $G$ has an IDS with at most $k$ vertices.

\vspace{0.05cm}

\textbf{The domination problem (DP)} \cite{booth1982}.

\emph{Input}: A graph $G$ and a positive integer $k$;

\emph{Property}: $G$ has an dominating set  with at most $k$ vertices.

\vspace{0.05cm}

\textbf{The $k$-rainbow independent domination problem ($k$RiDP)}.

\emph{Input}: A graph $G$ and two positive integers $k$ and $k'$;

\emph{Property}: $G$ has a $k$RiDF with weight at most $k'$.

\vspace{0.05cm}

The operation of \emph{identifying} two vertices $x$ and $y$ of a graph $G$ is to replace these vertices
by a single vertex incident to all the edges which were incident in $G$ to either $x$ or
$y$.

\begin{thm}\label{NP1}
The $k$-rainbow independent domination problem is $\mathcal{NP}$-complete for bipartite graphs.
\end{thm}
\begin{proof}
The $k$RiDP is a member of $\mathcal{NP}$, since we can check in polynomial time that a function from vertex set to $\{0,1, \ldots, k\}$ has weight at most $k'$ and is a $k$RiDF.

\begin{figure}[H]
  \centering
  \includegraphics[width=11.5cm]{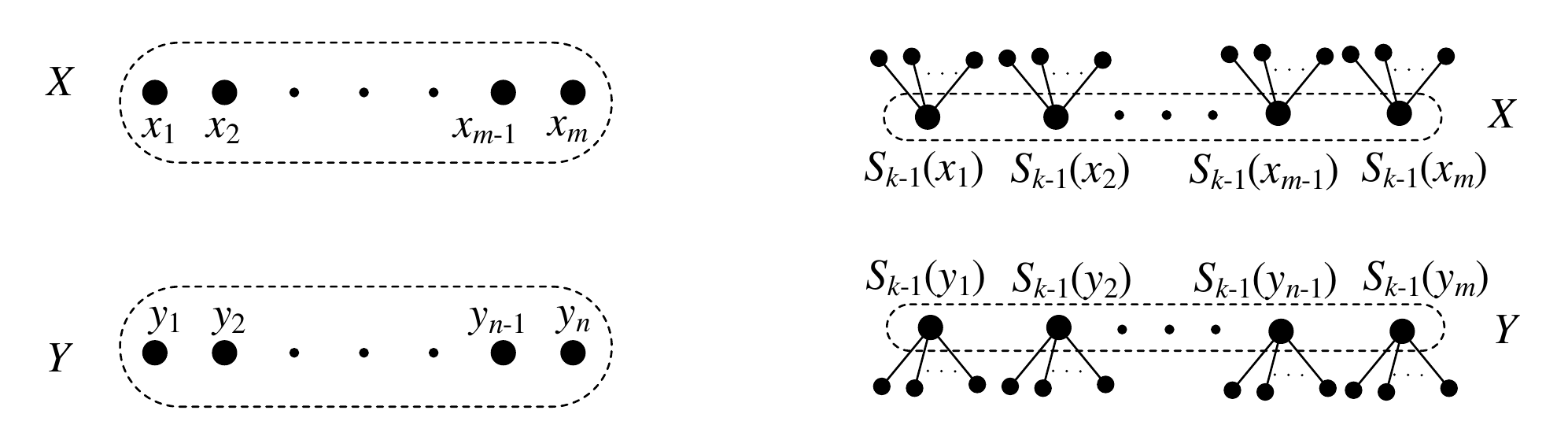}\\
  $G$\hspace{5cm} $G'$
  \caption{illustration of the construction from $G$ to $G'$}\label{fig1}
\end{figure}

When $k=1$, the $k$RiDP is equivalent to the IDP which is $\mathcal{NP}$-complete when $G$ is restricted to bipartite graphs \cite{goddard2013-0}. Therefore, we assume that $k\geq 2$. To show $\mathcal{NP}$-hardness, we give a reduction  from the domination problem (DP) for bipartite graphs, which is $\mathcal{NP}$-complete \cite{chang1984}.  Given a bipartite $G$ with a bipartition $(X,Y)$ where $X=\{x_1,x_2,\ldots, x_m\}$ and $Y=\{y_1,y_2,\ldots, y_n\}$, we construct a new graph $G'$ by adding $m+n$ copies of star $S_{k-1}$, denoted by $S_{k-1}(x_i)$ and $S_{k-1}(y_j)$ for $i\in [1,m]$ and $j\in [1,n]$, and identifying $w$ and the center of $S_{k-1}(w)$ for all $w\in \{x_i,y_j|i=1,\ldots,m, j=1,\dots, n\}$ (see Figure \ref{fig1}, in which we omit the edges between $X$ and $Y$). Clearly, $G'$ is also a bipartite graph.  We claim that $G'$ has a $k$RiDF with weight $(k-1)(m+n)+\ell$ if and only if $G$ has a dominating set of size $\ell$.

Given a $k$RiDF $f=(V_0,V_1,\ldots, V_k)$ of $G'$ with weight $(k-1)(m+n)+\ell$, let $D=(D'=V_1\cup \ldots \cup V_k)\cap (X\cup Y)$. Observe that all leaves of $S_{k-1}(x_i)$ and $S_{k-1}(y_j)$ for $i\in [1,m]$ and $j\in [1,n]$ belong to $D'$; therefore, $|D|=\ell$. Since $f$ is a $k$RiDF, it follows that every vertex in $V_0$ is adjacent to at least one vertex in $D$. Notice that $X\cup Y=V_0\cup D$; we see that $D$ is a dominating set of $G$.

Now, we  assume that $G$ has a dominating set $D$ where $|D|=\ell$. Let $D_1=D\cap X$ and $D_2=D\cap Y$. We define a function $f$: $V(G')\rightarrow [0,k]$ as follows: $f(v)=1$ for every $v\in D_1$, $f(v)=2$ for every $v\in D_2$ and $f(v)=0$ for every  $v\in (X\cup Y)\setminus D$. Since $G$ is  bipartite and $D$ is a dominating set of $G$, every vertex  $v\in (X\cup Y)\setminus D$ is adjacent to either $D_1$ or $D_2$ in $G$.
If $v$ is  adjacent to $D_i$ for some $i\in [1,2]$ in $G$, then we assign $[1,k]\setminus \{i\}$ to the $k-1$ leaves of $S_{k-1}(v)$ such that every leaf receives an unique number of $[1,k]\setminus \{i\}$.
Finally, for every $u\in D_i$ for $i=1,2$,  we assign $[1,k]\setminus \{i\}$ to the $k-1$ leaves of $S_{k-1}(u)$ such that every leaf receives an unique number of $[1,k]\setminus \{i\}$.
Clearly, $w(f)=|D|+(m+n)(k-1)$, $V_i$ is an independent set and every vertex in $V_0$ is adjacent to a vertex in $V_i$ for all $i\in [1,k]$. Therefore, $f$ is a $k$RiDF with weight $|\ell|+(m+n)(k-1)$. \qed
\end{proof}

\section{Conclusion}

In this paper, we respond some questions proposed by {\v{S}}umenjak et al. \cite{tadeja2018}, by proving an improved Nordhaus-Gaddum type inequality on  $k$-rainbow independent domination number and showing that the problem of deciding whether a graph has a $k$-rainbow independent dominating function of a given weight is $\mathcal{NP}$-complete. In the study, we proved that  when $G$ satisfies $\gamma_{\rm ri2}(G)=|V(G)|-1$ and $G\not\cong C_5$, it follows that $G$ is isomorphic to $S_n (n\geq 2)$, $S_n^+ (n\geq 2)$ or $S(n,1)(n\geq 1)$, and  $\gamma_{\rm ri2}(G)+\gamma_{\rm ri2}(\overline{G})=|V(G)|+2$. Additionally, we observe that $\gamma_{\rm ri2}(S(n,m)+\gamma_{\rm ri2}(\overline{S(n,m)})=|V(S(n,m))|+1$ when $m\geq 2$. Therefore, a question that arises is whether $S_n (n\geq 2)$, $S_n^+ (n\geq 2)$ and $S(n,1)(n\geq 1)$ are enough for determining graphs $G$ with the property of  $\gamma_{\rm ri2}(G)+\gamma_{\rm ri2}(\overline{G})=|V(G)|+2$. We formulate this more generally as follows:


\begin{question}
How to characterize graphs $G$ with $\gamma_{\rm ri2}(G)+\gamma_{\rm ri2}(\overline{G})=|V(G)|+2?$
\end{question}

\section*{Acknowledgments}
This work was supported by the National Natural Science Foundation of China (61872101, 61672051, 61309015, 61702075),  the China Postdoctoral Science Foundation under grant (2017M611223).





\section*{References}

 \bibliographystyle{elsarticle-num}
\bibliography{mybibfile}

\end{document}